\documentclass[11pt,reqno,fleqn,preprint,authoryear]{elsarticle}
\usepackage{amsmath, amsthm, amssymb}
\usepackage{verbatim,enumerate,graphicx}
\usepackage{hyperref}
\usepackage{subcaption}
\allowdisplaybreaks[2]
\usepackage[margin=1 in]{geometry}
\usepackage[margin=10pt,font=small,labelfont=bf,labelsep=endash]{caption}

\usepackage{booktabs}
\usepackage{multirow}
\usepackage{mathtools}

\numberwithin{equation}{section}
\newtheorem{thm}{Theorem}

\newtheorem{lem}{Lemma}

\theoremstyle{remark}

\theoremstyle{definition}

\newcommand{\1}[1]{\mathbbm{1}\!\left[#1\right]}
\newcommand{\Pb}[1]{\text{P}\!\left[#1\right]}

\newcommand{\deter}{\mathrm{det}}
\newcommand{\logit}{\mathrm{logit}}
\newcommand{\rone}{\mbox{${\mathbb{R}}$}}

\newcommand\myeq{\stackrel{\mathclap{\normalfont\mbox{$\mathcal{D}$}}}{=}}

\title{On the Domain of Attraction of a Tracy-Widom Law with Applications to Testing Multiple Largest Roots}

\author[hec]{Didier Ch\'etelat \fnref{fn1}}
\address[hec]{Department of Decision Sciences, HEC Montr\'eal}

\author[ashoka]{Rajendran Narayanan}
\address[ashoka]{Department of Mathematics, Ashoka University}

\author[cornell]{Martin T. Wells  \fnref{fn2}}
\address[cornell]{Department of Statistical Science, Cornell University}

\fntext[fn1]{Research was partially supported by NSF grant CCF-0808864.}
\fntext[fn2]{Research was partially supported by grants  NSF DMS-1208488 and NIH U19 AI111143. Corresponding author.}

\begin{document}
\begin{abstract}
\noindent   The greatest root statistic arises as the test statistic in several multivariate analysis settings. Suppose there is a global null hypothesis $H_0$ that consists of $m$ different independent sub null hypotheses, i.e., $H_0 = \cap_{k=1}^{m} H_{0k}$  and suppose the greatest root statistic is used as the test statistic for each sub null hypothesis. Such problems may arise when conducting a batch MANOVA or several batches of pairwise testing for equality of covariance matrices. Using the union-intersection testing approach and by letting the problem dimension $p \rightarrow \infty$ faster than $m \rightarrow \infty$ we show that $H_0$ can be tested using a Gumbel distribution to the approximate the critical values. Although the theoretical results are asymptotic, simulation studies indicate that the approximations are very good even for small to moderate dimensions. The results are general and can be applied in any setting where the greatest root statistic is used, not just for the two methods we use for illustrative purposes.\\

\smallskip
\noindent AMS 2010 subject classifications: 60G70, 62E20, 62H15.

\noindent Key words and phrases: Characteristic root, equality of covariance matrices, greatest root statistic, Gumbel distribution, MANOVA, multiple testing, Tracy-Widom laws, union-intersection test

\end{abstract}
\maketitle

\section{Introduction}\label{Intro}
\noindent Assuming the data generating process is multivariate Gaussian, the test statistics for hypotheses testing using the union-intersection approach arising in several multivariate analysis techniques is the largest eigenvalue of the multivariate beta distribution. More formally, suppose $X$ is an $n_1 \times p$ data matrix with each row being an independent copy of $N_p(0,\Sigma)$ then $A= X^TX \sim W_p(\Sigma,n_1)$ has a $p$ dimensional Wishart distribution with $n_1$ degrees of freedom. Let $B \sim W_p(\Sigma,n_2)$ be another Wishart distribution with $n_2$ degrees of freedom independent of $A$ with the same scale matrix $\Sigma$. If $n_1 > p$ then $A^{-1}$ exists and the non-zero eigenvalues of the matrix $A^{-1}B$ generalize the univariate $F$ statistic. The scale matrix has no effect on the distribution of these eigenvalues and so without loss of generality we can set $\Sigma = I_p$. The distribution of the random matrix $(A+B)^{-1}B$ is a generalization of the univariate beta distribution and is called the multivariate beta distribution or the Jacobi ensemble. The largest eigenvalue $\theta_{p}$ (also denoted by $\theta(p,n_1,n_2)$) of $(A+B)^{-1}B$ is a random variable called the {\it greatest root statistic} and since $A$ is positive definite $0 < \theta_{p} < 1$. We can also obtain $\theta_{p}$ as the largest root of the determinantal equation
$$
\deter [B - \theta (A+B)] = 0.
$$

\noindent The greatest root statistic arises as the null hypothesis distribution for the union-intersection test for several classical techniques such as MANOVA, test for equality of covariance matrices, canonical correlations and so on (see \cite{Muirhead}). \\

\noindent We consider the following problem. Suppose there is a global null hypothesis $H_0$ that consists of $m$ different independent sub null hypotheses, i.e., $H_0 = \cap_{k=1}^{m} H_{0k}.$  Such hypotheses arise when one is integrating data sets or assess effects across various treatment levels.   Consider a union-intersection type testing approach where the global null hypothesis is true if and only if each of the component sub null hypothesis is true. In such a setting the global null hypothesis would be rejected if the maximum of the test statistics arising from each sub null hypothesis falls in the appropriate rejection region. In particular, suppose the test statistic from each sub null hypothesis is the greatest root statistic, i.e., $\theta_{p,1},\theta_{p,2},\ldots,\theta_{p,m}$ where $\theta_{p,k}$ for each $k=1,2,\ldots,m$ is the greatest root statistic from the $k^{\text{th}}$ component sub null hypothesis. Then the decision rule to reject the global null hypothesis $H_0$ is, if the $\max\{\theta_{p,1},\theta_{p,2},\ldots,\theta_{p,m}\} > c$ for some appropriately chosen constant $c$. We show that the maximum of an i.i.d.\  sequence of the greatest root statistic falls in the Gumbel domain of attraction as $m \rightarrow \infty$ and hence the Gumbel distribution can be used to construct a test statistic to do inference for the global null hypothesis. Our approximation relies on two levels of asymptotics. The matrix dimension of each component multivariate beta distribution goes to infinity and also the number of sub null hypotheses under consideration goes to infinity but we let the matrix dimension go to infinity faster than the number of sub null hypotheses under consideration. In other words $p \rightarrow \infty$ faster than $m \rightarrow \infty$ in the sense to be precise made in Section \ref{Main}. \\

\noindent \cite{Dumitriu} review the fact that the exact null distribution of the greatest root statistic $\theta(p,n_1,n_2)$ is notoriously difficult to calculate. Deriving the exact distribution of the largest eigenvalue relies on performing a complicated $p-1$ dimensional integral with the Vandermonde term in the integrand. \cite{Constantine} showed that the marginal distribution of the largest eigenvalue can be expressed in terms of a hypergeometric function with a matrix argument. The cumulative distribution function of the greatest root statistic is
\noindent \begin{align}\label{hyp1}
 P(\theta_{p,1} < x) &=C_{1,p}  x^{\frac{pn_1}{2}}  {_2}F_1\Big(\frac{n_1}{2},\frac{-n_2+p+1}{2};\frac{n_1+p+1}{2};xI\Big) ,
\end{align}
where
\begin{equation*}
 C_{1,p}= \frac{\Gamma_p^{(1)}\Big(\frac{n_1+n_2}{2}\Big)\Gamma_p^{(1)}\Big(\frac{p+1}{2}\Big)}{\Gamma_p^{(1)}\Big(\frac{n_1+p+1}{2}\Big)\Gamma_p^{(1)}\Big(\frac{n_2}{2}\Big)}
\end{equation*}
and ${_2}F_1(\cdot,\cdot;\cdot,xI)$ denotes the hypergeometric function with a matrix argument, which in this case is considered to be the identity matrix. \cite{Gupta1985} gave exact Pfaffian expressions for hypergeometric functions with a matrix argument when the arguments are multiples of the identity matrix and also showed that the c.d.f.\ of the greatest root statistic can be expressed as a Pfaffian of a skew symmetric matrix whose entries are double integrals. \cite{Koev2006} exploit the recursion relations of Jack functions to develop efficient MATLAB implementations to evaluate the hypergeometric functions with a matrix argument.  More recently \cite{Butler2011} provide computational implementations of the theoretical framework advanced by \cite{Gupta1985}. \cite{Butler2011} express the double integrals of the Pfaffian in terms of series expansions that are computed using the Maple software.  There is an extensive literature on the algorithmic and computational aspects of dealing with the
hypergeometric functions with a matrix argument. An elegant treatment on the topic can be found in \cite{Dumitriu} and the references therein. \\

\noindent Moving away from the issue of computational techniques to evaluate the hypergeometric function with a matrix argument, in the remarkable paper of \cite{Johnstone2008}, it was shown that the greatest root statistic with suitable centering and scaling converges to the now ubiquitous Tracy-Widom distribution \cite{TW1994}, \cite{TW1996}. In particular, \cite{Johnstone2008} showed that assuming $p$ is even and that $p,n_{1}(p)$ and $n_{2}(p) \rightarrow \infty$ together in such a way that
\begin{align*}
  \lim\limits_{p \rightarrow \infty}\frac{\min(p,n_2)}{n_1 + n_2} > 0, \qquad \lim\limits_{p \rightarrow \infty} \frac{p}{n_1} < 1.
\end{align*}
Then the logit transform $T_p = \logit(\theta_{p}) = \log(\theta_{p}/1-\theta_{p})$ is approximately distributed according to the Tracy-Widom law, i.e.,
\begin{equation}
 \frac{T_p - \mu_p}{\sigma_p} \Rightarrow F_1
\end{equation}
where $F_1$ is the cdf of the Tracy-Widom distribution arising as a limiting distribution of the largest eigenvalue of Gaussian orthogonal ensembles and $\mu_p$ and $\sigma_p$ are centering and scaling factors to make the asymptotics work. We focus on the asymptotics as opposed to exact evaluation of the greatest root statistic owing to the second order rate of convergence $O(p^{-2/3})$ of the greatest root statistic to the Tracy-Widom law. As \cite{JohnstoneMa2012} show, this convergence rate can be guaranteed for appropriate centering and scaling factors and as illustrated by \cite{Johnstone2009} the Tracy-Widom approximation is fairly sharp even for small values of $p$ and works quite well for many applied data analysis questions. \\

The results are applicable in several multivariate analysis settings where the greatest root statistic plays a role. In particular  consider the following hypothesis testing framework to conduct pairwise testing of equality of covariance matrices arising from a multivariate normal sample. Let
$$
H_{01}:\Sigma_{11}=\Sigma_{12},\, H_{02}:\Sigma_{21}=\Sigma_{22}, \, \ldots, H_{0m}:\Sigma_{m1}=\Sigma_{m2}.
$$
Define the global hypothesis $H_0$ as $H_0 = \bigcap_{k=1}^m H_{0k}$.  Let $n_{k1},n_{k2}$ denote the sample sizes for the $k^{\text{th}}$ hypothesis test for $k=1,2,\ldots,m$ and let  $S_{k1},S_{k2}$ denote the covariance estimators for the $k^{\text{th}}$ hypothesis test.  Assuming that the underlying data generating process for each of the $m$ situations is a multivariate normal sample then under $H_{0k}$, $S_{k1} \sim W_p(\Sigma_k,n_{k1})$ and $S_{k2} \sim W_p(\Sigma_k,n_{k2})$ independent of
$S_{k1}$ where $\Sigma_k$ is the common covariance matrix under $H_{0k}$. Thus the test statistic for $H_{0k}$ is  $\theta_{p,k}$, which is the largest eigenvalue of $(n_{k1}S_{k1} + n_{k2}S_{k2})^{-1}n_{k2}S_{k2}$. Then $\max\{\theta_{p,1},\theta_{p,2},\ldots,\theta_{p,m}\}$ can be used to test$H_0$.  We will discuss this covariance testing problem in more detail in Section \ref{covmatrix}.\\

\noindent Our work is motivated to understand the bridge between two asymptotic regimes of extremes. From classical extreme value theory we know that the maximum of an i.i.d.\ sequence of random variables converges to one of three distributions depending on whether the random variables are light-tailed, heavy-tailed or have a finite support. For light-tailed random variables it is well known that the maximal domain of attraction is the Gumbel distribution and the Tracy-Widom distribution appears as the limiting distribution of random matrices with light-tailed i.i.d.\ entries. This prompts us to study the asymptotic maximal behaviour of i.i.d.\ extremal eigenvalues arising from a sequence of random matrices having light-tailed entries.

\section{Tracy Widom Distribution}
\noindent An important question of theoretical and practical interest is understanding the behavior of the largest eigenvalue of various classes of random matrices. If we consider a diagonal matrix with Gaussian entries then the largest eigenvalue of such a matrix would converge to the Gumbel distribution as the matrix dimension goes to infinity. This is because the maximal domain of attraction of the Gaussian distribution is the Gumbel distribution. However, when we consider a symmetric matrix with each entry being a real valued Gaussian random variable or a symmetric Hermitian random matrix with each entry being a complex valued Gaussian random variable then the largest eigenvalue converges to the Tracy-Widom distribution. It is indeed a remarkable fact that this distribution arises as the limiting distribution of a large class of random matrices and in fact the limit distribution of the largest eigenvalue has the Tracy-Widom law even if the assumption of i.i.d.\ Gaussian entries of the random matrix are
relaxed, see for example \cite{Soshnikov2002}. However, as shown in \cite{Soshnikov2006} when the matrix entries are heavy-tailed,  then the the joint distribution of the edge eigenvalues converge weakly to the inhomogeneous Poisson random point process. \\

\noindent Let $F_1$ denote the cumulative distribution function (cdf) of the Tracy-Widom distribution arising from the Gaussian orthogonal ensemble (GOE) and let $F_2$ be the cdf of the Tracy-Widom distribution arising from the Gaussian unitary ensemble (GUE) then from \cite{TW1994,TW1996} we know that
\begin{align}\label{TWUE}
F_2(x) &=\exp\left(-\int\limits\limits_{x}^{\infty}(y-x)q^2(y)dy\right)
\end{align}
and
\begin{align}\label{TWOE}
F_1(x) &= (F_2(x))^{\frac{1}{2}}\exp\left(-\frac{1}{2}\int\limits\limits_{x}^{\infty}q(y)dy\right),
\end{align}
where $q(x)$ is the solution of the classical Painlev\'{e} non-linear second order differential equation
\begin{align}
 q''(x) &= xq(x) + 2q^3(x),\qquad q(x) \sim \text{Ai}(x) \quad \text{as} \quad x \rightarrow \infty
\end{align}
and Ai$(x)$ denotes the Airy function.  \cite{Johnstone2001,Johnstone2008} demonstrated a universality property by showing that the largest eigenvalues of the Wishart matrix and the multivariate beta matrix both converge to the Tracy-Widom distribution, subject to some growth conditions on the size of the design matrix. \cite{NW2013} showed that the standardized maximum of an i.i.d.\ sequence of random variables having the Tracy-Widom distribution arising from the Gaussian unitary ensembles as in \eqref{TWUE} belongs to the Gumbel domain of attraction. \\
\begin{figure}[h!]
\centering
\begin{subfigure}{.5\textwidth}
  \centering
  \includegraphics[width=.9\linewidth]{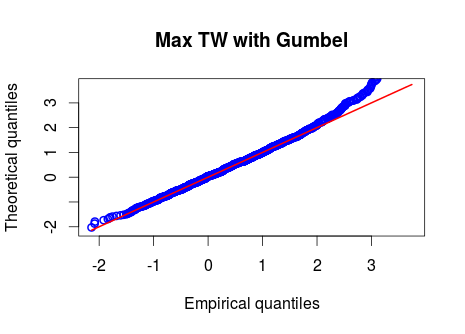}
  \caption{QQ Plot of Max TW with Gumbel}
  \label{fig:sub1}
\end{subfigure}%
\begin{subfigure}{.5\textwidth}
  \centering
  \includegraphics[width=.9\linewidth]{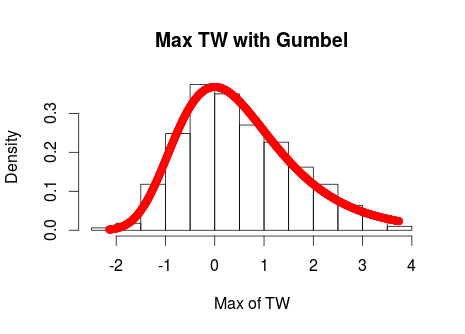}
  \caption{Histogram of Max TW with Gumbel}
  \label{fig:sub2}
\end{subfigure}
\caption{Simulated maximums of TW with Gumbel}
\label{fig:MaxTW_Gumbel}
\end{figure}

\noindent If we take an i.i.d.\ sequence of random variables having the Tracy-Widom (TW) distribution arising from the Gaussian orthogonal ensemble, as in, \eqref{TWOE} then the maximum of such a sequence asymptotically converges to the Gumbel distribution. (The authors discovered a crucial typo in one of the references while proving this result). Figure \ref{fig:MaxTW_Gumbel} \subref{fig:sub1} shows a QQ plot of simulated maximums of TW random variables and the standard Gumbel distribution based on $10000$ samples and Figure \ref{fig:MaxTW_Gumbel} \subref{fig:sub2} depicts a histogram of simulated maximums of TW random variables overlaid with a standard Gumbel distribution. A Kolmogorov-Smirnov test to check equality of normalized maximum of i.i.d.\ Tracy-Widom random variables with the Gumbel distribution fails to reject the null hypothesis at a p-value of $0.4658$.\\

\section{Main Result}\label{Main}
For every $p\geq1$ let $\theta_{1,1}(p,n_1,n_2),\ldots,\theta_{m,1}(p,n_1,n_2)$ denote an i.i.d.\ sequence of largest eigenvalues obtained from an i.i.d.\ sequence of multivariate beta random matrices $A_{p,1},\dots,A_{p,m}$ and $n_1 \geq p$. The meanings of $n_1$ and $n_2$ are as described in Section \ref{Intro}. Let $W_{p,k} = \logit\Big(\theta_{k,1}(A_{p,k})\Big)$ be their respective logit transformed largest eigenvalues. Denote by $F_1$ the cumulative distribution function of the Tracy-Widom distribution for the real case as in \eqref{TWOE}, and define the following normalization constants:
\begin{align}&
b_m=F_1^{-1}\Big(1-\frac1m\Big),
\qquad
a_m = 1/(mF_{1}'(b_m)).
\label{norm-consts}
\end{align}
Then from \cite{Johnstone2008} we know that
\begin{align}
\frac{T_{p,k} - \mu_{p, n_{1p}, n_{2p}}}{\sigma_{p, n_{1p}, n_{2p}}} \Rightarrow Z_1 \sim F_1,
\end{align}
where
\begin{align*}&
\mu_{p,n_{1p},n_{2p}}=2\log\tan\left(\frac{\phi_{p,n_{1p},n_{2p}}}2+\frac{\gamma_{p,n_{1p},n_{2p}}}2\right),
\\&
\sigma^3_{p,n_{1p},n_{2p}}=\frac{16/(n_{1p}+n_{2p}-1)^2}{\sin^2(\phi_{p,n_{1p},n_{2p}}+\gamma_{p,n_{1p},n_{2p}})\sin\phi_{p,n_{1p},n_{2p}}\sin\gamma_{p,n_{1p},n_{2p}}}
\end{align*}
and
\begin{align*}&
\phi_{p,n_{1p},n_{2p}}=2\arcsin\left(\frac1{\sqrt{2}}\sqrt{\frac{2\max(p,n_{1p})-1}{n_{1p}+n_{2p}-1}}\right),
\quad
\gamma_{p,n_{1p},n_{2p}}=2\arcsin\left(\frac1{\sqrt{2}}\sqrt{\frac{2\min(p,n_{1p})-1}{n_{1p}+n_{2p}-1}}\right).
\end{align*}
\begin{thm}\label{MEG}
Let $p, n_{1p}, n_{2p} ,m_p\rightarrow\infty$, with $n_{1p} \geq p$, $\lim_{p\rightarrow\infty}\min(p,n_{1p})/(n_{1p}+n_{2p})>0$ and $\lim_{p\rightarrow\infty}m_p/p^{2/3}<\infty$. Let $X^p_k$ denote the centred and scaled value obtained from the logit transform of the greatest root statistic,
\begin{align*}
X^p_k = \frac{T_{p,k} - \mu_{p,n_{1p},n_{2p}}}{\sigma_{p,n_{1p},n_{2p}}}.
\end{align*}
Then
\begin{align*}
Y^p=\frac{\underset{1\leq k\leq {m_p}}{\max} X^p_k-b_{m_p}}{a_{m_p}}\xrightarrow[p \rightarrow\infty]{\mathcal{D}}\text{\emph{Gumbel}}(0,1).
\end{align*}
\end{thm}

\noindent Before proving the main result, we present a few lemmas. Lemma \ref{maxtw} is the analogue of Theorem $1$ of \cite{NW2013}.

\begin{lem}\label{maxtw}
Let $Z_1, Z_2, \ldots $ be a sequence of i.i.d.\ random variables having the Tracy-Widom distribution arising from a Gaussian orthogonal ensemble (GOE) with cumulative distribution function $F_1$ as given in \eqref{TWOE}. Let $x^{\ast} = \sup\{x \in \rone | F_1(x) < 1\}$ denote the right end point of $F_1$. Here $x^{\ast} = \infty$. Then for $m_p\rightarrow\infty$ as $p\rightarrow\infty$, we have
\begin{align*}
\frac{{\underset{1\leq k\leq {m_p}}{\max}} Z_k-b_{m_p}}{a_{m_p}}\xrightarrow[p\rightarrow\infty]{\mathcal{D}}\text{\emph{Gumbel}}(0,1).
\end{align*}
\end{lem}
\begin{proof}
 We utilize Von Mises' condition to show the validity of our claim (the reader can refer to \cite{dehaan} or \cite{Resnick2008} for further details). Namely, if
\begin{align}\label{requiredlimit}
L(x) = \lim\limits_{x \rightarrow x^{\ast}}\frac{(1-F_1(x))F_{1}''(x)}{F_{1}'(x)^2} = -1,
\end{align}
then $F_1$ is in the domain of attraction of the Gumbel distribution. As a reminder, here $x^{\ast} = \infty$. To simplify calculations, we obtain the following from \cite{TW1}
\begin{align*}&
E(x)=\exp\left(-\frac12\int_x^\infty q(s)ds\right),
\quad
F(x)=\exp\left(-\frac12\int_x^\infty (s-x)q^2(s)ds\right).
\end{align*}
Observe that $F_1(x) = F(x)E(x)$ and $F_2(x) = F^2(x)$ and it can be easily seen that
\begin{align*}&
E'(x)=\frac{E(x)}2q(x),
\quad
F'(x)=\frac{F(x)}2\int_x^\infty q^2(s)ds.
\end{align*}
Therefore $F'_1(x)=F_1(x)R_1(x)$ where
$R_1(x)=\frac12\left[q(x)+\int_x^\infty q^2(s)ds\right]$.
\begin{align*}
L(x)=\frac{\big[1-F_1(x)\big]F''_1(x)}{F_1'(x)^2}=\frac{1-F_1(x)}{F_1(x)}+\frac{1-F_1(x)}{F_1(x)}\frac{R_1'(x)}{R_1^2(x)}.
\end{align*}
We are interested in finding $\lim\limits_{x\rightarrow\infty}L(x)$. From section $1.1.1$ of \cite{TW1} it follows that
\begin{align*}
F_1(x)=1-\left[
\frac{e^{-\frac23x^{3/2}}}{4\sqrt{\pi} x^{3/4}}
+\frac{e^{-\frac43x^{3/2}}}{32\pi x^{3/2}}
-\frac{e^{-2x^{3/2}}}{128\pi^{3/2} x^{9/4}}
\right]\bigg(1\!+\!O(x^{-\frac32})\bigg).
\end{align*}
There is a typographical error in \cite{TW1} in their expression for this expansion, there should be $x^{3/4}$ in the denominator.  If we take this in account, we can write
\begin{align}\label{asymptoneminusF1}
1-F_1(x)\sim\frac{e^{-\frac23x^{3/2}}}{4\sqrt{\pi} x^{3/4}},
\end{align}
From \cite{Bassom1998} and lemma $3$ of \cite{NW2013} we get,
\begin{align*}
\frac12\left[q(x)+\int_x^\infty q^2(s)ds\right]
=\frac{e^{-\frac23x^{3/2}}}{4\sqrt{\pi}x^{1/4}}\bigg(1\!+\!O(x^{-\frac32})\bigg)
+
\frac{e^{-\frac43x^{3/2}}}{16\pi x}\bigg(1\!+\!O(x^{-\frac32})\bigg)
\end{align*}
which yields the following asymptotic expression for $R_1(x)$:
\begin{align}\label{asymptR1}
R_1(x)\sim \frac{e^{-\frac23x^{3/2}}}{4\sqrt{\pi}x^{1/4}}.
\end{align}
Again using the asymptotic expansion from \cite{Bassom1998} we get
\begin{align*}
\frac12\left[q'(x)-q^2(x)\right]
=
-\frac{x^{1/4}e^{-\frac23x^{3/2}}}{4\sqrt{\pi}}\bigg(1\!+\!O(x^{-\frac32})\bigg)
-\frac{e^{-\frac43x^{3/2}}}{8\pi x^{1/2}}\bigg(1\!+\!O(x^{-\frac32})\bigg).
\end{align*}
This yields
\begin{align}\label{asympR1prime}
R_1'(x)\sim-\frac{x^{1/4}e^{-\frac23x^{3/2}}}{4\sqrt{\pi}}.
\end{align}
Since $F_1(x)$ is a cdf, $F_1(x) \sim 1$ as $x \rightarrow \infty$, so
\begin{align*}
\frac{1-F_1(x)}{F_1(x)}\frac{R_1'(x)}{R_1^2(x)}\sim -1
\end{align*}
as $x\rightarrow\infty$. Thus $\lim\limits_{x\rightarrow\infty}L(x)=-1$ which establishes that the maximum of an i.i.d.\ sequence of Tracy-Widom distribution from GOE is in the Gumbel domain of attraction
\begin{align*}
F_\text{Gumbel(0,1)}(x)=\exp(-e^{-x}).
\end{align*}
Therefore, for the normalizing constants we defined, we have
\begin{align*}
\frac{{\underset{1\leq k\leq {m_p}}{\max}} Z_k-b_{m_p}}{a_{m_p}}\xrightarrow[p\rightarrow\infty]{\mathcal{D}}\text{Gumbel}(0,1)
\end{align*}
as desired.
\end{proof}

\begin{lem}\label{lem:ABA} We have
\begin{align*}&
a_m\sim\left[\frac43\right]^{\frac12}\left[\frac34\right]^{\frac16}\log^{-1/3}\left(\frac{m^2}{12\pi}\right),
\qquad
b_m\sim\left[\frac34\right]^{\frac23}\log^{2/3}\left(\frac{m^2}{12\pi}\right).
\end{align*}
as $m\rightarrow\infty$. Then for $m\rightarrow\infty$ and fixed $y\in\mathbb{R}$ one has
$\lim\limits_{m\rightarrow\infty}a_m=0$ and $\lim\limits_{m\rightarrow\infty}b_m=\infty$.


\end{lem}
\begin{proof}
As stated earlier $b_m = U(m)$ where $U(m)$ is the left continuous inverse of $1/(1-F_1)$. Thus using \eqref{asymptoneminusF1} we can write
\begin{align*}&
m=\frac1{1-F_1(b_m)}\sim h(b_m)=4\sqrt{\pi}b_m^{3/4}\exp\left(\frac23b_m^{3/2}\right).
\end{align*}
Let $g(x)=\log h(e^x)=\log(4\sqrt{\pi})+\frac34x+\frac23\exp\left(\frac32x\right)$. For any $y \in \rone$,
\begin{align*}
\left|\frac{d g^{-1}(y)}{dy}\right|=\left|\frac1{g'(g^{-1}(y))}\right|=\frac1{\frac34+\exp\left(\frac32g^{-1}(y)\right)}\leq \frac43.
\end{align*}
Then, by the mean value theorem
\begin{align*}&
\left|\log h^{-1}(m)-\log b_m\right|
=
\left|g^{-1}(\log m)-g^{-1}\circ g(\log b_m)\right|
\\&\qquad=\;
\left|g^{-1}(\log m)-g^{-1}\left(\log h(b_m)\right)\right| \\&\qquad\leq\;
\frac43\left|\log(m)-\log h(b_m)\right|  \;\xrightarrow{x\rightarrow\infty} 0.
\end{align*}
Therefore
\begin{align*}
b_m\sim h^{-1}(m)=\left[\frac34 W\left(\frac{m^2}{12\pi}\right)\right]^{2/3}\sim\left[\frac34\log\left(\frac{m^2}{12\pi}\right)\right]^{2/3},
\end{align*}
where $W$ is the Lambert W function. Second, defining $d_m=\exp(b_m^{3/2})$, we also have
\begin{align*}
m=\frac1{1-F_1(\log^{2/3}d_m)}\sim \tilde h(d_m)=4\sqrt{\pi}\log^{1/2}(d_m) d_m^{2/3}.
\end{align*}
Now for $\tilde g(x)=\log \tilde h(e^x)=\log(4\sqrt{\pi})+\frac12\log x+\frac23x$, which is invertible as a function $[0,\infty)\rightarrow\mathbb{R}$,
\begin{align*}
\left|\frac{d g^{-1}(y)}{dy}\right|=\left|\frac1{g'(g^{-1}(y))}\right|=\frac1{\frac1{2g^{-1}(y)}+\frac23}\leq \frac32
\end{align*}
for any $y\in\mathbb{R}$. Then, by the mean value theorem
\begin{align*}&
\left|\log \tilde h^{-1}(m)-\log d_m\right|
=
\left|g^{-1}(\log m)-g^{-1}\circ g(\log d_m)\right|
\\&\qquad=\;
\left|g^{-1}(\log m)-g^{-1}\left(\log \tilde h(d_m)\right)\right| \\&\qquad\leq\;
\frac32\left|\log(m)-\log \tilde h(d_m)\right|  \;\xrightarrow{m\rightarrow\infty} 0.
\end{align*}
So
\begin{align*}
\exp(b^{3/2}_m)=d_m\sim\tilde h^{-1}(m)=\exp\left(\frac34W\left(\frac{m^2}{12\pi}\right)\right)\sim\left[\frac{m^2/12\pi}{\log(m^2/12\pi)}\right]^{3/4}.
\end{align*}
Therefore,
\begin{align*}&
a_m=\frac1{mF_1'(b_m)}=\frac1{mF_1(b_m)R_1(b_m)}
\sim
\frac{4\sqrt{\pi}b_m^{1/4}\exp\left(\frac23b_m^{3/2}\right)}{m}
\\&\qquad\sim\frac{4\sqrt{\pi}\left[\frac34\log\left(\frac{m^2}{12\pi}\right)\right]^{1/6}\left[\frac{m^2/12\pi}{\log(m^2/12\pi)}\right]^{1/2}}{m}
\sim\frac{\left[\frac43\right]^{1/2}\left[\frac34\right]^{1/6}}{\log^{1/3}(m^2/12\pi)},
\end{align*}
as desired.
\end{proof}

\noindent We now prove Theorem \ref{MEG}.

\begin{proof}
We use \cite[Theorem 1]{Johnstone2008}. The conditions required are that
\begin{align*}
\lim_{p\rightarrow\infty}\frac{\min(p,n_{1p})}{n_{1p}+n_{2p}}>0,
\qquad
\lim_{p\rightarrow\infty}\frac{p}{n_{2p}}<1,
\end{align*}
which are satisfied by the assumptions of the theorem. Then $\mu_{p,n_{1p},n_{2p}}$ and $\sigma_{n,p}$ are defined as in Equation (5) on p. 2641, and so, under the null hypothesis, by \cite[Theorem 1]{Johnstone2008} with $s_0=0$ there must be a $C>0$ such that
$$
\Big|\Pb{X^p_k\leq x}-\Pb{Z_k\leq x}\Big|\leq\frac{C}{p^{2/3}}e^{-x/2}
$$

\noindent for all $x\geq 0$. Second, for any fixed $y\in\mathbb{R}$, $\lim_{p\rightarrow\infty}a_{m_p}y+b_{m_p}=\infty$ by Lemma \ref{lem:ABA}, so there is some $P(y)>0$ such that for all $p\geq P(y)$, $a_{m_p}y+b_{m_p}>0$. Then, for $Z_1,Z_2,...$ a sequence of independent real Tracy-Widom random variables, we have for all $p\geq P(y)$
\begin{align*}&
\Big|\Pb{Y^p\leq y}-\exp\left(-e^{-y}\right)\Big|
\notag\\&\qquad\leq
\left|\Pb{\underset{1\leq k\leq m_p}{\max} X^p_k\leq a_{m_p}y+b_{m_p}}
-\Pb{\underset{1\leq k\leq {m_p}}{\max} Z_k\leq a_{m_p}y+b_{m_p}}\right|
\notag\\&\hspace{80pt}+\left|\Pb{\underset{1\leq k\leq {m_p}}{\max} Z_k\leq a_{m_p}y+b_{m_p}}-\exp\left(-e^{-y}\right)\right|
\notag\\&\qquad\leq
\sum_{k=1}^{m_p}\left|\prod_{l=1}^{k-1}\Pb{Z_l\leq a_{m_p}y+b_{m_p}}\prod_{l=k}^{m_p}\Pb{X^p_l\leq a_{m_p}y+b_{m_p}}
\right.\notag\\&\hspace{80pt}\left.-\prod_{l=1}^{k}\Pb{Z_l\leq a_{m_p}y+b_{m_p}}\prod_{l=k+1}^{m_p}\Pb{X^p_l\leq a_{m_p}y+b_{m_p}}\right|
\notag\\&\hspace{60pt}+\left|\Pb{\underset{1\leq k\leq {m_p}}{\max} Z_k\leq a_{m_p}y+b_{m_p}}-\exp\left(-e^{-y}\right)\right|
\notag\\&\qquad\leq
{m_p}\Big|\Pb{X^p_1\leq a_{m_p}y+b_{m_p}}-\Pb{Z_1\leq a_{m_p}y+b_{m_p}}\Big|
\\&\hspace{60pt}+\left|\Pb{\underset{1\leq k\leq {m_p}}{\max} Z_k\leq a_{m_p}y+b_{m_p}}-\exp\left(-e^{-y}\right)\right|
\notag\\&\qquad\leq
C\frac{{m_p}}{p^{2/3}}e^{-\frac{1}{2}(a_{m_p}y+b_{m_p})}
+\left|\Pb{\underset{1\leq k\leq {m_p}}{\max} Z_k\leq a_{m_p}y+b_{m_p}}-\exp\left(-e^{-y}\right)\right|.
\end{align*}
Thus since $\lim_{p\rightarrow\infty}m_p/p^{2/3}<\infty$ and $\lim_{p\rightarrow\infty}a_{m_p}y+b_{m_p}=\infty$, using lemma \ref{maxtw} we get
\begin{align*}&
\lim_{p\rightarrow\infty}\Big|\Pb{Y^p\leq y}-\exp\left(-e^{-y}\right)\Big|
\notag\\&\hspace{80pt}\leq
\lim_{p\rightarrow\infty}\left|\Pb{\underset{1\leq k\leq {m_p}}{\max} Z_k\leq a_{m_p}y+b_{m_p}}-\exp\left(-e^{-y}\right)\right|
\notag\\&\hspace{80pt}\leq
0.
\end{align*}
Since this is true for any $y\in\mathbb{R}$, the result follows.
\end{proof}

\subsection{Approximate $\alpha$ level test}\label{covmatrix}
\noindent As a motivating example from multivariate analysis, consider the following hypothesis testing framework to conduct pairwise testing of equality of covariance matrices arising from a multivariate normal sample. Let
$$
H_{01}:\Sigma_{11}=\Sigma_{12},\, H_{02}:\Sigma_{21}=\Sigma_{22}, \, \ldots, H_{0m}:\Sigma_{m1}=\Sigma_{m2}.
$$
Define the global hypothesis $H_0$ as $H_0 = \bigcap_{k=1}^m H_{0k}$. This implies that $H_0$ is true if and only if each of the component hypothesis $H_{0k}$ is true. Thus, we {\it accept} $H_0$ if and only if every component hypothesis $H_{0k}$ is accepted. We can equivalently say that we reject $H_0$ if any component hypothesis $H_{0k}$ is rejected.

Let $R_k$ denote the rejection region corresponding to the $k^{\text{th}}$ hypothesis test, so that $R = \bigcup_{k=1}^m R_k$ is the rejection region corresponding to $H_0$. Let $n_{k1},n_{k2}$ and $S_{k1},S_{k2}$ denote the sample sizes and covariance estimators, respectively, for the $k^{\text{th}}$ hypothesis test, where $k=1,2,\ldots,m$. By construction, $S_{k1}$ and $S_{k2}$ will be independent. If we further assume that each of the $m$ samples follow a multivariate normal distribution, then under $H_{0k}$ we would have $S_{k1} \sim W_p(\Sigma_k,n_{k1})$ and $S_{k2} \sim W_p(\Sigma_k,n_{k2})$, where $\Sigma_k$ is the common covariance matrix under $H_{0k}$. Thus the test statistic for $H_{0k}$ is  $\theta_{p,k}$, which is the largest eigenvalue of $(n_{k1}S_{k1} + n_{k2}S_{k2})^{-1}n_{k2}S_{k2}$. Then $\max\{\theta_{p,1},\theta_{p,2},\ldots,\theta_{p,m}\} \Rightarrow G_p$ as $m \rightarrow \infty$ where $G_p$ denotes the cdf of a univariate Gumbel distribution, where we explicitly write the dependence
on the dimension $p$.

Using this, we can construct an approximate, high-dimensional $\alpha$-level test for $H_0: \Sigma_{1k}=\Sigma_{2k},\;\forall k=1,...,m$ using the union-intersection approach. Indeed, we could reject $H_0$ when $\max\{\theta_{p,1},\theta_{p,2},\ldots,\theta_{p,m}\} > c_{\alpha}$,
where
\begin{align*}
c_\alpha=\left[1+\exp\bigg(\sigma_{p,n_{1p},n_{2p}}a_{m_p}\log\big(-\log[1-\alpha]\big)
-\sigma_{p,n_{1p},n_{2p}}b_{m_p}-\mu_{p,n_{1p},n_{2p}}\bigg)\right]^{-1}.
\end{align*}
This would an approximate $\alpha$-level test in the sense that for $p_n/n\rightarrow (0,\infty)$ and $m_p/p^{2/3}\rightarrow(0,\infty)$,
\begin{align*}
\lim_{p\rightarrow\infty}\Pb{\text{Reject }H_0\,\vert\, H_0}=\alpha.
\end{align*}
To see this, note that in the notation of Theorem \ref{MEG},
\begin{align*}&
\Pb{\text{Reject }H_0\,\vert\, H_0}=\Pb{\underset{1\leq k\leq {m_p}}{\max} \theta_{p,1} > c_\alpha\,\bigg\vert\, H_0}
\\&\quad=
\text{P}\left[
\frac{\underset{1\leq k\leq {m_p}}{\max}\text{logit}\,\theta_{p,1}\!\left([n_{k1}S_{k1}+n_{k2}S_{k2}]^{-1}S_{k2}\right)-\mu_{p, n_{1p}, n_{2p}}}
{\sigma_{p, n_{1p},n_{2p}} }
\right.\\&\hspace{160pt}>\left.
\vphantom{\frac{\underset{1\leq k\leq {m_n}}{\max}\text{logit}\,\lambda_1\!\left([W_{1k}+W_{2k}]^{-1}W_{2k}\right)-\mu_{p, n_{1p},n_{2p}}}{\sigma_{p, n_{1p},n_{2p}} }}
\frac{\text{logit}\,c_\alpha-\mu_{p, n_{1p},n_{2p}}}{\sigma_{p, n_{1p},n_{2p}}}\,\bigg\vert\, H_0
\right]
\\&\quad=
\Pb{\underset{1\leq k\leq {m_p}}{\max}X^p_k > \frac{\text{logit}\,c_\alpha-\mu_{p, n_{1p},n_{2p}}}{\sigma_{p, n_{1p},n_{2p}}}\,\bigg\vert\, H_0}
\\&\quad=
\Pb{Y^p> \frac{\text{logit}\,c_\alpha-\mu_{p, n_{1p},n_{2p}}-\sigma_{p, n_{1p},n_{2p}}b_{m_p}}{\sigma_{p, n_{1p},n_{2p}}a_{m_p}}\,\bigg\vert\, H_0},
\end{align*}
so according to this same theorem it would hold that
\begin{align*}
\lim_{p\rightarrow\infty}\Pb{\text{Reject }H_0\,\vert\, H_0}&=
1-\exp\left(-\exp\left(-\frac{\text{logit}\,c_\alpha-\mu_{p,n_{1p},n_{2p}}-\sigma_{p,n_{1p},n_{2p}}b_{m_p}}{\sigma_{p,n_{1p},n_{2p}}a_{m_p}}\right)\right)
\\&=\alpha,
\end{align*}
as wanted.

\section{Simulation}
To explore the finite $(m, n, p)$ behavior of our theoretical domain of attraction results we carry out two numerical studies in this section.  We consider two different large-scale inferential problems: pairwise testing for equality of covariance matrices and  multivariate analysis of variance. In each simulation setting, we compute the power curves for different dimensions over one-dimensional spaces of alternatives.

\subsection{Equality of Covariance Matrices}
The theory behind this test was discussed in Subsection \ref{covmatrix}. We have $m$ independent population pairs. For the $k^{\text{th}}$ population pair $(k=1,2,\ldots,m)$, let $k_1$ be the index of the first population in the $k^{\text{th}}$ pair and $k_2$ be the index of the second population in the $k^{\text{th}}$ pair. Let $n_{k1}$ and $n_{k2}$ be the sample sizes of the first and the second population in the $k^{\text{th}}$ pair. Let $\Sigma_{k1}, \Sigma_{k2}$ be the corresponding covariance matrices for the $k^{\text{th}}$ pair.

We simulated two independent $p$ dimensional multivariate normal data sets that form the two design matrices of dimensions $n_{k1}\times p$ and $n_{k2}\times p$ respectively. The test statistic to test the $k^{\text{th}}$ null hypothesis is the largest eigenvalue $\theta_{p,k}$ of $(n_{k1}S_{k1} + n_{k2}S_{k2})^{-1}n_{k2}S_{k2}$ where $S_{k1}$ and $S_{k2}$ are the sample covariance matrix analogues of $\Sigma_{k1}$ and $\Sigma_{k2}$ respectively. \\

\noindent We considered two different regimes for generating covariance matrix pairs that need to be tested for equality. In the first regime, for each $k=1,\ldots,m$ we set $\Sigma_{k1} = I_p$ and $\Sigma_{k2}=\gamma I_p$, where $I_p$ denotes the $p$ dimensional identity matrix and $\gamma \in [1,2.5]$ is a non-negative scalar giving rise to a one parameter family of alternatives. We then performed $8000$ simulations to test for simultaneous equality of $m=500$ covariance matrix pairs for each value of $\gamma$ in the grid. We repeated the exercise for matrix dimensions ranging from $p=10$ to $100$,  while the sample sizes for each pair were chosen as $n_1=n_2=p/2$. We then computed the resulting approximations to the true power curves. The results for $p=10,30,70,100$ are plotted on Figure \ref{fig:PowerCovEq}.\\

\begin{figure}[h!]
\centering
\includegraphics[width=8.2cm,height=8.2cm]{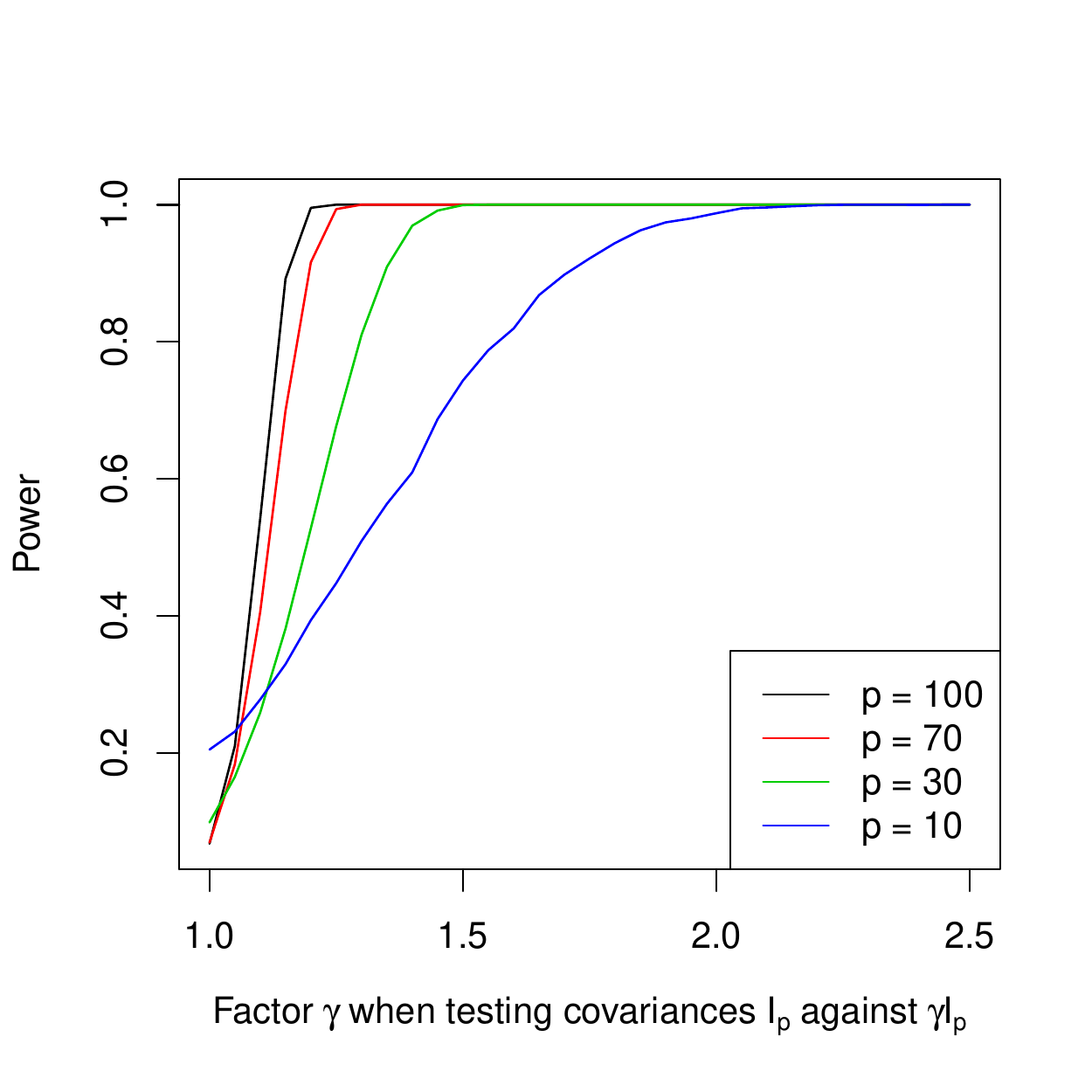}
\caption{Power curves for simultaneous covariance equality tests}
\label{fig:PowerCovEq}
\end{figure}

\noindent  Note that when $\gamma = 1$, both the null and the alternative hypothesis represent the identity matrix. As can be seen from Figure \ref{fig:PowerCovEq}, our approximation does very well in detecting departures from the null hypothesis even for very small values of $p$ and $\gamma$. The power curve approaches $1$ very quickly and gets much sharper even for a moderate values of $p$ and mild increase of $\gamma$ from $1$ . This supplements the theoretical asymptotic results rather well.

\subsection{MANOVA}

\noindent Our second set-up involved $m$ independent batches. Within each batch, we had $r$ different groups each of which contained $n$ i.i.d.\ samples from a $p$-dimensional normal distribution. Between groups of the same batch, we had equal covariances but potentially unequal means. \\
\begin{center}\begin{tabular}{ccc}
$Y_{111}, \cdots, Y_{11n} \sim N_p(\mu_{11},\Sigma_1),$ &
&
$Y_{m11}, \cdots, Y_{m1n} \sim N_p(\mu_{1r},\Sigma_m),$ \\
$\cdots$ &
$\cdots$ &
$\cdots$ \\
$\underbrace{Y_{1r1}, \cdots, Y_{1rn} \sim N_p(\mu_{11},\Sigma_1),}$ &
&
$\underbrace{Y_{mr1}, \cdots, Y_{mrn} \sim N_p(\mu_{1r},\Sigma_m).}$ \\
Batch $1$ &
&
Batch $m$ \\
\end{tabular}\end{center}
We wanted to test the global null hypothesis of equality of group means across independent batches,
\begin{center}
	\begin{tabular}{rc}
		\multirow{3}{*}{$H_0$:} & ${\bf \mu_{11}}=\cdots={\bf \mu_{1r}},$ \\
		& $\cdots$ \\
		& ${\bf \mu_{m1}}=\cdots={\bf \mu_{mr}}.$
	\end{tabular}
\end{center}
\noindent It is to be emphasized that each row in the above null hypothesis expression is a $p$ dimensional vector.
For each batch $1\leq k\leq m$, we computed the matrices
\begin{align*}
A_k = \sum_{l=1}^r\sum_{i=1}^n (Y_{kli}-\bar{Y}_{kl})(Y_{kli}-\bar{Y}_{kl})',
\qquad
B_k = n\sum_{l=1}^r (\bar{Y}_{kl}-\bar{Y}_k)(\bar{Y}_{kl}-\bar{Y}_k)',
\end{align*}
where
\begin{align*}&
\bar{Y}_{kl} = \frac1n\sum_{i=1}^nY_{kli},
\qquad
\bar{Y}_k = \frac1r\sum_{l=1}^p\bar{Y}_{kl}.
\end{align*}
That is, for the $k^{\text{th}}$ batch, $A_k$ was the {\it within group} covariance matrix and $B_k$ was the {\it between group} covariance matrix. Under the null hypothesis, we had $A_k\sim\text{W}_p(r(n-1),\Sigma_k)$ independent of $B_k\sim\text{W}_p(r-1,\Sigma_k)$ so that
\begin{center}
\begin{tabular}{c}
$\theta_1 = \lambda_1([A_1+B_1]^{-1}B_1)\sim \theta_{1,1}(p,r(n-1),r-1\big)$ \\
$\cdots$ \\
$\theta_m = \lambda_1([A_m+B_m]^{-1}B_m)\sim \theta_{m,1}(p,r(n-1),r-1\big)$
\end{tabular}
\end{center}
where $p$ refers to the dimension, $r(n-1)$ refers to the ``error'' degrees of freedom and $r-1$ is the ``hypothesis'' degrees of freedom for each batch. Furthermore, $\theta_1,\ldots,\theta_m$ were independent because the batches were independent. Consider the following argument: write $n_1=r-1$ and $n_2=r(n-1)$, and suppose that for fixed $n$,  $p,r,m\rightarrow\infty$ with
$\lim_{p\rightarrow\infty}m/p^{2/3}<\infty$ and $\lim_{p\rightarrow\infty}p/r>0$. Then $n_{1}$ and $n_{2}\rightarrow\infty$ and
\[\lim_{p\rightarrow\infty}\frac{\min(p,n_{1})}{n_{1p}+n_{2}}
=\lim_{p\rightarrow\infty}\frac{\min\Big(\frac{p}{r},1-\frac1{r}\Big)}{n-\frac1{r}}
=\frac{\min\Big(\lim\limits_{p\rightarrow\infty}\frac{p}{r},1\Big)}{n}>0.\]\\
Then, according to Theorem \ref{MEG}, we would find that
\begin{align*}
Z=\frac{
\underset{1\leq k\leq m}{\max}\text{logit}\big(\theta_k\big)
-\mu_{p,r-1,r(n-1)}-b_{m}\sigma_{p,r-1,r(n-1)}
}{a_{m}\sigma_{p,r-1,r(n-1)}}
\xrightarrow[p \rightarrow\infty]{\mathcal{D}}\text{Gumbel}(0,1),
\end{align*}
where $a_m$, $b_m$ are defined as Equation \eqref{norm-consts}. Hence, an approximate $\alpha$-test for testing $H_0$ could be given by rejecting when $Z>F_{\text{Gumbel}(0,1)}^{-1}(1-\alpha)$. As an aside, in some situations it could be convenient to work with the following reparametrization outlined in \cite{MKB}:
\begin{align*}
\theta_{k,1}\Big(p,r(n-1),r-1\Big) &\myeq \theta_{k,1}\Big(r-1,r(n-1)+r-1-p,p\Big) \\
  &\myeq \theta_{k,1}\Big(r-1,rn-1-p,p\Big).
\end{align*}
It can be easily shown that the asymptotic regime and hence the simulation results are invariant under the above reparametrization. \\

\noindent Now in order to generate the power curves for our hypothesis testing framework, we tested against the one-parameter family of alternatives
\begin{center}
\begin{tabular}{rc}
\multirow{3}{*}{$H_1(\gamma)$:} &
$(\mu_{11},...,\mu_{1r})
=(\begin{bmatrix}1^\gamma \\ \dots \\1^\gamma\end{bmatrix},
\begin{bmatrix}2^\gamma \\ \dots \\2^\gamma\end{bmatrix},
...,
\begin{bmatrix}r^\gamma \\ \dots \\r^\gamma\end{bmatrix})$ \\
& $\cdots$ \\
& $(\mu_{m1},...,\mu_{mr})
=(\begin{bmatrix}1^\gamma \\ \dots \\1^\gamma\end{bmatrix},
\begin{bmatrix}2^\gamma \\ \dots \\2^\gamma\end{bmatrix},
...,
\begin{bmatrix}r^\gamma \\ \dots \\r^\gamma\end{bmatrix}).$
\end{tabular}
\end{center}
for $\gamma\in[0,1]$. We performed $8000$ simulation runs for each $p,r,\gamma$ combination. This was done $p$ ranging from $p=10$ to $p=100$, and for each such choice of $p$ we set $r=2p$. We then computed approximations to the true power curves based on these simulations. The results for $p=10,30,70,100$ are plotted on Figure \ref{fig:PowerMANOVA}.\\
\begin{figure}[h!]
\centering
\includegraphics[width=8.2cm,height=8.2cm]{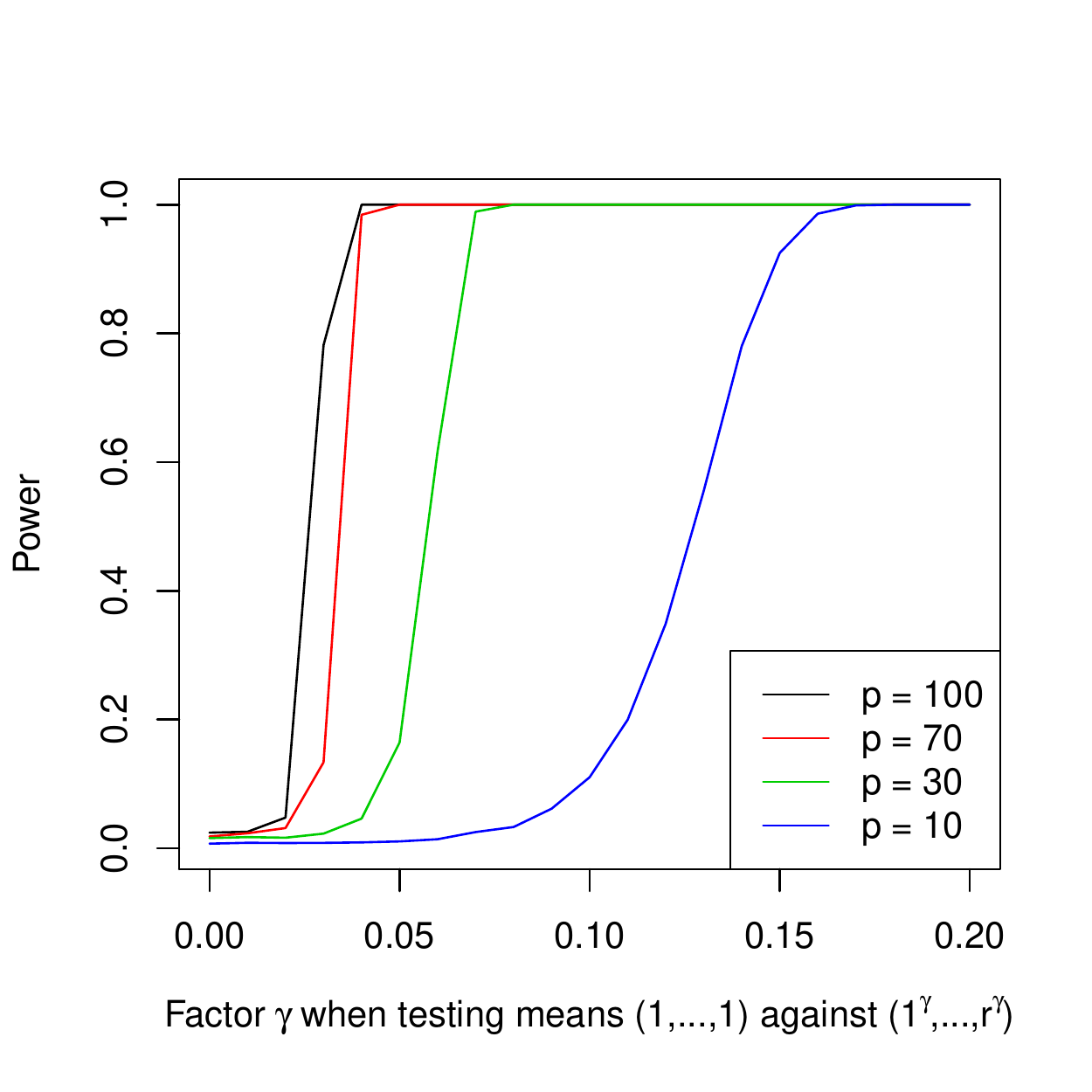}
\caption{Power curves for simultaneous MANOVA tests}
\label{fig:PowerMANOVA}
\end{figure}

\noindent Note that when $\gamma=0$, the alternate hypothesis and the  null hypothesis coincided. Just like in the first simulation setting, the results are good even for small to moderate values of $p$ and for very mild departures from the null hypothesis, as evidenced by small positive values of $\gamma$ yielding power close to $1$. Further, as expected, the power curves even steeper as the problem dimension increases from $p=10$ to $p=100$. This is in agreement with our theoretical findings.

\section{Discussion}

The greatest root statistic arises as the test statistic in several multivariate statistical analysis settings. We explored the problem of several independent multivariate analysis testing problems when each hypothesis instance is the greatest root statistic. It is not difficult to fathom casting batch MANOVA or batch pairwise testing for equality of covariance matrices in our hypothesis testing framework. In this article we prove that the maximal domain of attraction of an i.i.d.\ sequence of greatest root statistics arising out of such batch testing settings is the Gumbel distribution. We present the efficacy of the asymptotic results through two canonical multivariate analysis techniques. \\

The results in this article are quite general and can, in principle, be applied to any situation where several independent instances of the greatest root statistics are employed as the test statistic. In particular, one can recast the underlying model in this article as array data where the $m$ dimension represents the various faces of the arrays.  Array variate random variables are mainly useful for multiply labeled random variables that can naturally be arranged in array form. Some examples include response from multi-factor experiments, two-three dimensional image and video data, spatial-temporal data, repeated measures data.  The methods of this article can be used to test homogeneity across the faces of the array. \\

\section{\refname}
\bibliographystyle{elsarticle-num-names}
\bibliography{GumbelGreatestRoot_JMA}
\end{document}